\documentclass[11pt]{amsart}
\usepackage[all,cmtip]{xy}            
\usepackage{comment}
\usepackage{graphicx}
\usepackage{amssymb}
\usepackage{epstopdf}
\excludecomment{va}

\title{K-Semistability for irregular Sasakian manifolds}
\author{Tristan C. Collins}
\author{G\'abor Sz\'ekelyhidi}
\theoremstyle{plain}
\newtheorem{thm}{Theorem}
\newtheorem{prop}{Proposition}[section]
\newtheorem{defn}{Definition}[section]
\newtheorem{lem}{Lemma}[section]
\newtheorem{cor}{Corollary}

\theoremstyle{remark}

\newcommand{\del}{\partial} 
\newcommand{\dbar}{\overline{\del}}
\newcommand{\ddb}{i\del\dbar}

\newcommand{\C}{\mathbb{C}}
\newcommand{\mt}{\mathfrak{t}}

\newcommand{\Spec}{\text{Spec }}

\begin{document}

\maketitle 
\begin{abstract}
	We introduce a notion of K-semistability for Sasakian manifolds. This
	extends to the irregular case the orbifold K-semistability of
	Ross-Thomas. Our main result is
	that a Sasakian manifold with constant
	scalar curvature is necessarily K-semistable.   As an application, we show how one can recover the
volume minimization results of Martelli-Sparks-Yau, and the
Lichnerowicz obstruction of Gauntlett-Martelli-Sparks-Yau from this
point of view.
\end{abstract}

\section{Introduction}

Determining necessary and sufficient conditions for the existence of
constant scalar curvature K\"ahler (cscK) metrics on a compact K\"ahler manifold
$X$ is an important open problem in complex geometry, requiring methods
from algebraic geometry and partial differential equations.  When
$c_{1}(X)$ represents a negative or trivial cohomology class then
Yau~\cite{Y2} proved
that there exist K\"ahler-Einstein metrics 
in classes proportional to $c_1(X)$ (see also Aubin~\cite{Aubin} for the
$c_1(X) < 0$ case). However when
the first Chern class is postive, or we are looking at general K\"ahler
classes, then there are obstructions to existence.  
A famous conjecture of Yau  says that the
existence of K\"ahler-Einstein metrics, when $c_1(X) > 0$, 
should be equivalent to some geometric invariant theory (GIT) notion of
stability for the underlying variety \cite{Y}. Tian~\cite{Tian}
introduced the notion of K-stability, and showed that it is a necessary
condition for existence. The notion of K-stability was refined by
Donaldson~\cite{D2}, and extended to any K\"ahler class given by the
first Chern class $c_1(L)$ of an ample line bundle $L$. The
Yau-Tian-Donaldson conjecture states that K-stability (or some
refinement of it) of a polarized
manifold $(X,L)$ is equivalent to the existence of a cscK metric in
$c_1(L)$. For more on this very active area of
research, see the survery of Phong-Sturm~\cite{PSstab}, 
and the references therein.

From the works of Donaldson~\cite{D}, Stoppa~\cite{Sto} and
Mabuchi~\cite{Mab2} we know that the existence of a cscK metric implies
K-stability. Part of this work was subsequently
generalized by Ross-Thomas~\cite{RT} to the case of orbifolds 
with cyclic quotient
singularities, where they showed that the existence of a cscK orbifold
metric in $c_{1}(L)$ implies the K-\emph{semi}stability of the polarized
orbifold $(X,L)$.  As we will recall below, the work of
Ross-Thomas can be phrased as a result about quasi-regular Sasakian manifolds,
and the goal of the present paper is to extend this to the
irregular case. 

Sasakian geometry is an important, odd-dimensional conterpart of K\"ahler
geometry.    Recently, the subject of Sasakian geometry has garnered a
great deal of interest partly due to its connections to the AdS/CFT
correspondence in theoretical physics \cite{JMM}, which provides a
detailed correspondence between certain conformal field theories and
geometries (see also \cite{Keh,KW, AFHS, MP}).  
It is thus a natural problem to determine
necessary and sufficient conditions for the existence of Sasaki-Einstein
metrics, and more generally, Sasakian metrics of constant scalar curvture
(or even the analogs of extremal metrics \cite{BGS}).
As in the K\"ahler case, this is well understood when the basic first
Chern class is negative or zero \cite{EKA, SWZ}.  However, when the
basic first  Chern class is positive, there are obstructions to
existence \cite{FOW, GMSY, MSY}.  It is expected that a suitably
generalized version of the conjecture of Yau \cite{Y} should hold.  Some
sufficient conditions for the existence of Sasaki-Einstein metrics have
been provided in \cite{TCC} in the spirit of \cite{PS, PSSW}, by
examining conditions which are sufficient to guarantee the convergence
of the Sasaki-Ricci flow, introduced in \cite{SWZ}.  However, these
conditions are not obviously GIT related.  Recent work on obstructions
to the existence of Sasaki-Einstein metrics has provided the first hints
of a connection to stability for Sasakian manifolds; see, for example,
\cite{MSY, MSY1} and the references therein.  In \cite{MSY} the authors
show that a Sasakian manifold admits a Sasaki-Einstein metric only when
the Reeb vector field minimizes the volume functional.  Moreover, when
the minimizing Reeb vector field is quasi-regular, they show that volume
minimization is equivalent to the vanishing of the classical Futaki
invariant on the  quotient orbifold.  This is a remarkable result,
translating a problem which seems algebro-geometric in nature into a
dynamical problem which can be studied analytically.

In the current work, our primary interest is to develop a notion of
K-stability for a Sasakian manifold $S$. When the Sasakian manifold is
quasi-regular, then this is equivalent to the work of Ross-Thomas~\cite{RT} on
K-stability for orbifolds. The question of whether there is a suitable
extension of this to the irregular case has been posed in several places
in the literature \cite[Problem 7.1]{S}, \cite{MSY, FO}.
In this paper we provide such an
extension, and prove that the
existence of a constant scalar curvature Sasakian metric implies the
K-semistability of $S$. Our main result is thus the following corollary of
Theorem~\ref{Don thm}. 

\begin{cor}\label{main cor} Let $(S,g)$ be a Sasakian manifold with Reeb
	vector field $\xi$.  If $g$ has constant scalar curvature, then
	the cone $(C(S), \xi)$ is K-semistable.  
\end{cor}
 
As already suggested by Sparks~\cite{S}, the obvious approach is to
approximate a given irregular Sasakian manifold with a sequence of regular
ones, and attempt to take a limit of the obstructions provided by the
results of Ross-Thomas in the orbifold case. The main difficulty with this
is that when the Reeb vector field is irrational, then it is not clear how
to generalize the usual setup of K-stability to the non-Hausdorff quotient
space. If instead we were to work directly with the odd dimensional Sasakian
manifolds, then the methods of algebraic geometry are no longer available
to define test-configurations and their Futaki invariants. Because of this
we work on
the cone over the Sasakian manifold, which is an affine variety with an
isolated singular point. These affine varieties come with the action of a
torus $T$,
and a Reeb field on the Sasakian manifold corresponds to an element in the
Lie algebra of $T$. We can naturally define test-configurations for affine
varieties with a $T$-action, and for any Reeb field $\xi$ on the underlying
Sasakian manifold we define a Futaki invariant of such a
test-configuration. The crucial points are that this Futaki invariant (for
a given test-configuration) varies continuously with the Reeb field $\xi$,
and when $\xi$ is rational, then it matches up with Ross-Thomas's Futaki
invariant for test-configurations in the orbifold case. We can then
use an approximation argument to extend Ross-Thomas's results to the case
of irrational Reeb fields. 

Our definition of the Futaki invariant uses the Hilbert series as opposed
to the usual Riemann-Roch expansions. This point of view was already used
in \cite{MSY} in the form of the index character, to compute the volume of
a Sasakian manifold. The main advantage is that for orbifolds the
Riemann-Roch expansions contain periodic terms, which become unmanageable
as we approach an irrational Reeb field, whereas in the index character
these periodic terms are not visible. This allows us to show the continuity
of the Futaki invariant with respect to the Reeb field. 

As an application, we recover the results of \cite{MSY} algebraically
by showing that
in the situation they considered, volume minimization is equivalent to
K-semistability for product test configurations. As a second
application, we show that the Lichnerowicz obstruction to existence of
Sasaki-Einstein metrics  studied in \cite{GMSY} can be interpreted in
terms of K-semistability for deformations arising from the Rees algebra
of a principal ideal of the coordinate ring.  Indeed, for rational Reeb
fields, the Lichnerowicz obstruction was interpreted in terms of slope
stability for the quotient orbifold in \cite{RT}.  Our computations
recover this result, and extend it to irrational Reeb fields by
establishing an explicit formula for the Donaldson-Futaki invariant of
the Rees deformation.  It is worth pointing out that both 
volume minimization and the Lichnerowicz
obstruction has a natural interpretation in terms of 
the AdS/CFT dual field theory;  see \cite{MSY1, GMSY} for details.  This
presents an interesting opportunity to attempt to recover the necessity
of K-semistability for existence of Sasaki-Einstein metrics as the geometric
dual of a physical constraint in the dual conformal field theory.
Such a connection would be beneficial from both a physical and
mathematical standpoint.

We begin our developments in section 2 with a brief review of Sasakian
manifolds, and  affine schemes polarized by a Reeb vector field.  We also briefly recall some facts about orbifolds and orbifold K-stability .  In section 3 we define the
Calabi functional on a polarized affine variety equipped with a K\"ahler
metric compatible with the Reeb field.  In section 4 we discuss the
index character and the Donaldson-Futaki invariant of a polarized affine
variety.  In section 4 we define test configurations for polarized
affine varieties, and K-semistability.  We then use the results of
sections 2, 3 and 4 to prove our main theorem.  Finally, in section 6 we
show that the volume minimization results of Martelli, Sparks and Yau
\cite{MSY}, and the Lichnerowicz obstruction of \cite{GMSY} arise from
K-stability considerations for product test configurations, and the Rees
algebra respectively.

\vspace{\baselineskip}
\noindent
{\bf Acknowledgements} The first author would like to thank his advisor
Professor D.H. Phong for his support and encouragement, as well as Zack
Maddock, You Qi, and Michael McBreen for many helpful conversations. The
second author would like to thank Professor S.T. Yau for introducing him
to Sasakian manifolds, and J. Ross for helpful discussions on
orbifold K-stability.

\section{Background}
\subsection{Sasakian geometry}\label{sec:background}
In this section we will recall some aspects of Sasakian geometry that we
will use. 
There are various points of view on the subject, and for a thorough
treatment see Boyer-Galicki~\cite{BG}. 

\begin{defn} A Sasakian manifold is a smooth Riemannian manifold $(S,g)$,
	$\dim_{\mathbb{R}}S =2n+1$, such that the metric cone $(C(S),
	\bar{g}) := (S \times \mathbb{R}_{\geq 0}, dr^{2} + r^{2}g)$ is
	K\"ahler.  Note that $S$ is canonically imbedded in $C(S)$ as
	the set $\{r=1\}$.  
\end{defn}

A Sasakian manifold inherits a number of key properties from its
K\"ahler cone.  In particular, an important role is played by the Reeb
vector field.  
\begin{defn} The Reeb vector field is $\xi =
	J(r\partial_{r})|_{\{r=1\}}$, where $J$ denotes the integrable
	complex structure on $C(S)$.  
\end{defn}

The Reeb vector field is a unit length, real holomorphic, Killing vector
field whose integral curves foliate $S$ by geodesics.  Sasakian manifolds
are roughly categorized by their Reeb vector fields.  When the intregral
curves of $S$ are all compact, the action of the Reeb vector field
integrates to a $U(1)$ action.  A Sasakian manifold is said to be
\emph{regular} if this $U(1)$ action is free; otherwise, it is said to
be \emph{quasi-regular}.  When the integral curves of the Reeb vector
field are not all compact, the Sasakian manifold is said to be
\emph{irregular}.  The regular and quasi-regular Sasakian manifolds are
well understood, owing to the following theorem of Boyer and Galicki.

\begin{thm}[\cite{BG1} Theorem 2.4]\label{embedding} Let $S$ be a
	compact regular or quasi-regular Sasakian manifold.  The space
	of leaves of the Reeb foliation $Z$ is a K\"ahler manifold or
	orbifold, respectively.  Moreover, $Z$ is a normal, projective,
	$\mathbb{Q}$-factorial algebraic variety.  
\end{thm}

The results of Rukimbira \cite{R} imply that any irregular Reeb vector
field can be approximated by quasi-regular Reeb fields.  In particular,
every Sasakian manifold admits at least one quasi-regular Reeb vector
field.  Combining this with Theorem~\ref{embedding}, we see that for any
Sasakian manifold $S$, the cone over $S$ is an affine variety with
an isolated singularity at $0$.  With this observation in hand, for the
remainder of this paper, we will work primarily with affine varieties,
smooth away from $0$, and comment on the Sasakian aspects of our work
only where pertinent. 

When defining test-configurations, we will need to consider degenerations
of an affine variety into possibly non-reduced schemes, and we will need an
algebro-geometric formulation of the notion of a Reeb vector field. 
Suppose that $Y \subset \C^{N}$ is an affine scheme, with a torus $T
\subset Aut(Y)$. Let us write $\mt := \text{Lie}(T_\mathbb{R})$ for the Lie
algebra of the maximal compact sub-torus.  Let
$\mathcal{H}$ denote the global sections of the structure sheaf of $Y$,
and write
\begin{equation*} \mathcal{H} = \oplus_{\alpha \in \mt^{*}}
	\mathcal{H}_{\alpha}
\end{equation*}
for the weight decomposition under the action of $T$. 

\begin{defn}\label{Reeb field defn} A vector $\xi \in \mt$ is a
	\emph{Reeb vector field} if 
	for each non-empty weight space $\mathcal{H}_{\alpha}$, with
	$\alpha\not=0$, we have
	$\alpha(\xi) > 0$, i.e. $\xi$ acts with positive weights on the
	non-constant functions on $Y$.  We will often 
	identify the vector $\xi$
	with the vector field it induces on $Y$.  We define the
	\emph{Reeb cone} to be
	\begin{equation*} \mathcal{C}_{R} :=\{ \xi \in \mt \big| \xi
		\text{ is a Reeb field } \} \subset \mt. 
	\end{equation*}
	Since $\mathcal{H}$ is finitely generated, $\mathcal{C}_R$
	is a rational, convex, polyhedral cone, and for any $\xi\in
	\mathcal{C}_R$ there is an $\epsilon > 0$ such that $\alpha(\xi)
	\geqslant \epsilon |\alpha|$ for all non-empty weight spaces. 
	We say that $\xi$ is rational if there exists $\lambda
	\in \mathbb{R}_{>0}$ such that $\alpha(\lambda \xi) \in
	\mathbb{N}$ for every non-empty weight space.
	Otherwise, we say that $\xi$ is irrational.  
\end{defn}

Note that any homogeneous variety admits a Reeb field generated by the
usual $\C^{*}$ action on $\C$.  In analogy with this case, we shall call
an affine scheme $Y$ with a holomorphic torus action admitting a Reeb
vector field a \emph{polarized affine scheme}.  An affine scheme $Y$ may
admit more than one Reeb field; choosing a Reeb vector field $\xi$ is
analogous to fixing a polarization for a projective scheme.  For the
most part, we shall consider only polarized affine \emph{varieties}.
The next lemma shows that Reeb vector fields are always induced from Lie
algebra actions on the ambient space, possibly after increasing the
codimension of the embedding.

\begin{lem}\label{lin act} Let $Y \subset \C^{N}$ be an affine scheme,
	and let $T$ be a torus acting holomorphically on $Y$.  Then
	there exists an embedding $Y \hookrightarrow \C^{N'}$ and a
	torus $T' \subset GL(N',\C)$ such that the multiplicative action
	of $T'$ on $\C^{N'}$ induces the action of $T$ on $Y$.
\end{lem} 
\begin{proof} Let $Y$ be cut out by the ideal $I \subset
	\C[x_{1},\dots,x_{N}]$, so that $Y = \Spec \mathcal{H}$ for
	$\mathcal{H}= \C[x_{1},\dots,x_{N}]/I$.  The torus $T$ induces a
	decomposition 
	\begin{equation*} \mathcal{H} = \oplus_{\alpha \in
		\mt^{*}} \mathcal{H}_{\alpha},
	\end{equation*} 
	and
	the images of $x_{1},\dots,x_{n}$ generate
	$\mathcal{H}$.  In particular, there exists a finite set
	of homogenous generators $u_{1},\dots u_{N'} \in
	\mathcal{H}$, with weights
	$\alpha_{1},\dots,\alpha_{N'}$.  Consider the map
	\begin{equation*} \begin{array}{ccc}
		\C[x_{1},\dots,x_{N'}]& \longrightarrow
		&\mathcal{H} \\ x_{i} &\longmapsto & u_{i}
	\end{array} \end{equation*} 
	Define an action of $T$ on
	$\C[x_{1},\dots,x_{N'}]$, where $T$ acts on $x_{i}$ with
	weight $\alpha_{i}$.  We get an exact sequence
	\begin{equation*} 0 \longrightarrow I' \longrightarrow
		\C[x_{1},\dots,x_{N'}] \longrightarrow
		\mathcal{H} \longrightarrow 0 
	\end{equation*}
	which is equivariant with respect to the torus
	action.  We obtain 
	\begin{equation*} 
		\Spec \mathcal{H} \cong \Spec
		\frac{\C[x_{1},\dots,x_{N'}]}{I'}
		\hookrightarrow \Spec
		\C[x_{1},\dots,x_{N'}], 
	\end{equation*}
	and hence an embedding $Y
	\hookrightarrow \C^{N'}$.  The action of
	$T$ on $Y$ is induced by the linear,
	diagonal action of $T$ on $\C^{N'}$ as
	desired.  
\end{proof}
 
Because of this lemma, we are essentially dealing with affine schemes
defined by ideals $I\subset\mathbb{C}[x_1,\ldots,x_N]$ for some $N$, which
are homogeneous for the action of a torus $T\subset GL(N,\mathbb{C})$. We
can even assume that the torus action is diagonal. A choice of an integral
vector
$\xi\in\mt$ then induces a grading on $\mathbb{C}[x_1,\ldots,x_N]$, which
has positive weights when $\xi$ is a Reeb vector. 

We will now relate our algebraic Reeb cone to the one defined differential
geometrically in \cite{MSY} (see also He-Sun~\cite{HS}, and the Sasaki Cone
in \cite{BGS}). 
Suppose that $Y\subset\mathbb{C}^N$ is an affine variety, smooth away from
the origin, and $Y$ is defined by an ideal
$I\subset\mathbb{C}[x_1,\ldots,x_N]$, homogeneous for the diagonal action
of a torus $T$. We will also assume that $Y$ is not contained in a linear
subspace. 

\begin{defn}
A K\"ahler metric $\Omega$ on $Y$ is \emph{compatible}
with a Reeb vector field $\xi\in\mt$ if there exists a $\xi$-invariant
function $r: Y \rightarrow \mathbb{R}_{>0}$ such that $\Omega =
\frac{1}{2}\ddb r^{2}$ and $\xi = J (r\frac{\del}{\del r})$, where $J$
denotes the restriction of the complex structure of $Y$.
\end{defn}

Fixing a Reeb field, and a compatible metric is analogous to fixing an
ample line bundle $L$, and choosing a metric in $c_{1}(L)$.  To see
this, let $Y$ be a polarized affine variety with $\dim_{\C}Y=n+1$, and
let $\xi$ be a rational Reeb vector field. Let $\xi_{\C}$ be the
complexification of $\xi$ and consider the holomorphic action induced by
$\xi_{\C} \in \mt_{\C}$.  Then $Y\backslash\{0\}$ is a principal
$\C^{*}$ orbibundle over the orbifold $X = Y/ \C^{*}$ corresponding to
an ample orbi-line bundle $L \rightarrow X$.  In particular,
$Y\backslash\{0\}$ is the complement of the zero section in the total
space of the orbi-line bundle $L^{-1}$.   By the Kodaira-Bailey
embedding theorem \cite{Ba}, the ampleness of $L$ is equivalent to the
existence of a Hermitian metric $h$ on $L^{-1}$ such that $\omega =\ddb
\log h$ is a metric on $X$.  We define a function $r:Y\rightarrow
\mathbb{R}_{>0}$ by $(z, \sigma) \rightarrow |\sigma|_{h(z)}$, for
$\sigma$ in the fibre of $L^{-1}$  over $z \in X$.  We get a metric on
$Y$ by setting 
\begin{equation}\label{metric}
\Omega = \ddb r^{2}. 
\end{equation}
In particular, when $\xi$ is rational, $(Y,\xi)$ always admits a
compatible K\"ahler metric.

Given a rational Reeb vector $\xi_0$, and compatible metric $\Omega_0$ on $Y$, the contact
1-form $\eta_0$ is defined to be dual to $\xi_0$. The Reeb cone is defined
in \cite{HS} to be
\begin{equation}\label{eq:RCone}
	\mathcal{C}'_R = 
	\{ \xi\in\mt\,|\, \eta_0(\xi) > 0\,\text{ on } Y\setminus\{0\}\}. 
\end{equation}
\begin{prop}
	The cone $\mathcal{C}'_R$ in 
	\eqref{eq:RCone} above coincides with the Reeb cone
	$\mathcal{C}_R$ that we defined in Definition \ref{Reeb field defn}. 
\end{prop}
\begin{proof}
	We need to relate the condition that $\eta_0(\xi) > 0$ with the
	weights of the circle action generated by $\xi$ on the ring of
	functions. As shown in \cite{MSY}, $H = \frac{1}{2}
	r^2\eta_0(\xi)$ is a Hamiltonian for the vector field $\xi$ with
	respect to $\Omega_0$. It follows that 
	\[ J\xi = -\nabla H, \]
	and moreover $H\to 0$ as we approach the cone point $0$. 
	
	Suppose first that $H$ is
	strictly positive, so $\xi$ cannot vanish anywhere. It follows
	that if we write $\phi_t : Y\to Y$ for the negative gradient flow
	of $H$, then 
	\[ \lim_{t\to\infty} \phi_t(p) = 0,\]
	for any $p\in Y$. Suppose that $f$ is a non-constant
	regular function on $Y$ (for instance a coordinate function on the
	ambient $\mathbb{C}^N$), on
	which $\xi$ acts with weight $\lambda$, and $p$ is a point such
	that $f(p)\not=0$. Then
	$J\xi(f) = -\lambda f$, 
	so 
	\[ \frac{d}{dt} f(\phi_t(p)) = -\lambda f(\phi_t(p)). \]
	Since $f(0)=0$, we must have $\lambda > 0$. So if $\xi\in
	\mathcal{C}'_R$, then $\xi\in\mathcal{C}_R$. 

	Conversely suppose that $H$ is negative somewhere. Since $H$ is
	homogeneous under $r\frac{\partial}{\partial r}$, we can then find
	points arbitrarily close to $0$, where $H$ is negative. For a
	suitable point $p$, the \emph{positive} gradient flow $\phi_t$ of
	$H$ will satisfy $\phi_t(p)\to 0$ as $t\to\infty$. Then the same
	argument as above shows that if $f$ is a non-constant
	homogeneous function for
	$\xi$ which does not vanish at $p$, then the weight of $\xi$ on $f$
	must be negative. 
\end{proof}

\begin{cor}\label{approx cor}
If $\xi$ is an irrational Reeb vector field on $Y$ and $\Omega$ is a
compatible K\"ahler metric, then there exists a
sequence $\xi_{n}\mt$ of rational Reeb vector fields and compatible metrics
$\Omega_k$ on $Y$, such that $\xi_k\to \xi$ in $\mt$, and the $\Omega_k$
converge to $\Omega$ smoothly on compact subsets of $Y$.
\end{cor}
\begin{proof}
	Since $\mathcal{C}_R$ is a rational convex polyhedral cone, we can
	approximate $\xi$ with a sequence of rational elements
	$\xi_k\in\mathcal{C}_R$. The construction of \cite[Lemma 2.5]{HS}
	shows that there are compatible K\"ahler metrics $\Omega_k$, which
	will converge to $\Omega$ on compact subsets of $Y$. 
\end{proof}

\subsection{Orbifold K-stability}\label{sec:orb}

For a review of the basic properties of orbifolds with constant scalar
curvature metrics in mind, see Ross-Thomas \cite{RT1}. Similarly to
them, we will only be
interested in polarized orbifolds, and as explained in \cite{RT1} Remark
2.16., these can be viewed as global $\mathbb{C}^*$-quotients of affine
schemes. More precisely, given a finitely generated graded ring
\[ R=\bigoplus_{k\geqslant 0} R_k\] 
over $\mathbb{C}$, the grading induces a
$\mathbb{C}^*$-action on $\mathrm{Spec}(R)$. When $\mathrm{Spec}(R)$ is
smooth, the corresponding
orbifold is the quotient of $\mathrm{Spec}(R)\setminus\{0\}$ by this
$\mathbb{C}^*$-action. More generally the quotient is a Deligne-Mumford
stack. In our terminology below, the grading
corresponds to a choice of rational Reeb field on the affine scheme
$\mathrm{Spec}(R)$. 

Differential geometrically the affine scheme $Y=\mathrm{Spec}(R)$, if
smooth away from the origin, arises
as the blowdown of the zero section in the total space of $L^{-1}$ for
an orbifold $X$ with orbiample line bundle $L$. In Section~\ref{sec:cal}
below we
will express the Calabi functional on the orbifold $X$ in terms of a
cone metric on $Y$. In the rest of this section we will review the work
of \cite{RT1} which gives a lower bound for the Calabi functional on
the orbifold $X$ in terms of the Futaki invariants of
test-configurations.  

Roughly speaking a 
test-configuration for a polarized orbifold $(X,L)$ is a polarized,
flat, $\mathbb{C}^*$-equivariant family over $\mathbb{C}$, whose generic
fiber is $(X,L^r)$ for some $r > 0$. In greatest generality the family
should be allowed to be a Deligne-Mumford stack. For computations it is useful to
reformulate this more algebraically. Let 
\[ R=\bigoplus_{k\geqslant 0} H^0(X,L^r) \] 
be the homogeneous coordinate ring of $(X,L)$. Any set of
homogeneous generators $f_1,\ldots,f_k$ of $R$ give rise to an embedding
of $X\hookrightarrow\mathbb{P}$ into a weighted projective space. 
Assigning weights to the
$f_1,\ldots, f_k$ induces a
$\mathbb{C}^*$-action on the weighted projective space $\mathbb{P}$.
Acting on $X\hookrightarrow\mathbb{P}$ we obtain a family
$X_t\subset\mathbb{P}$ for $t\not=0$. Taking the flat completion of this
family across $t=0$ is a test-configuration $\chi$. The central fiber of this
test-configuration is a polarized Deligne-Mumford stack $(X_0, L_0)$,
with a $\mathbb{C}^*$-action. It is convenient to allow $L_0$ to be a
$\mathbb{Q}$-line bundle, so that on the generic fiber we recover $L$
instead of a power of $L$. Let us write $d_k = \dim H^0(L_0^k)$ and
let $w_k$ be the total weight of the $\mathbb{C}^*$-action on
$H^0(L_0^k)$. As explained in \cite{RT1}, the
Riemann-Roch theorem from To\"en~\cite{Toe} implies that for large $k$ we
have expansions
\begin{equation}\label{eq:expand} \begin{aligned}
	d_k &= a_0 k^n + (a_1 + \rho_1(k))
	k^{n-1} + \ldots, \\
	w_k &= b_0 k^{n+1} + (b_1 +
	\rho_2(k))k^n+ \ldots,
\end{aligned}
\end{equation}
where $\rho_1,\rho_2$ are periodic functions with average zero. The
Futaki invariant of the test-configuration is then defined to be
\[ Fut(\chi) = \frac{a_1}{a_0}b_0 - b_1. \]
Writing $A_k$ for the infinitesimal generator of the
$\mathbb{C}^*$-action on $H^0(L_0^k)$, there is also an expansion
\[ \mathrm{Tr}(A_k^2) = c_0 k^{n+2} + O(k^{n+1}), \]
and the norm of the test-configuration is defined by
\[ \Vert\chi\Vert^2 = c_0 - \frac{b_0^2}{a_0}. \]

The main result that we need is the extension by Ross-Thomas~\cite{RT1}
of Donaldson's lower bound for
the Calabi functional~\cite{D}, to orbifolds. 
\begin{thm}[Donaldson, Ross-Thomas] 
	Suppose that $(X,L)$ is a polarized orbifold of dimension $n$, 
	and let $\omega\in
	c_1(L)$ be an orbifold metric. In addition suppose that $\chi$
	is a test-configuration for $(X,L)$. Then 
	\[ \Vert\chi\Vert\cdot \Vert R_\omega - \hat{R}\Vert_{L^2(\omega)}
	\geqslant -c(n)Fut(\chi), \]
	where $R_\omega$ is the scalar curvature of $\omega$, and
	$\hat{R}$ is its average. 
\end{thm}
Although this result is not stated explicitly in \cite{RT1}, it follows
easily from their proofs. In particular in their Theorem 6.6 the
constant $C$ can be taken to be $\frac{1}{2}\left(\mathrm{vol}\sum_i
c_i\right)^{1/2}$, while in the proof of Theorem 6.8 the constant $c$
equals $a_0\sum_{i} c_ik^{n+1}$ to highest order. Combining these, the
last inequality in the proof of Theorem 6.8 gives the result we need.

\section{The Calabi Functional on a Polarized Affine
Variety}\label{sec:cal}
Let us suppose as in Section~\ref{sec:background} that $Y\setminus\{0\}$ is
the complement of the zero section in the total space of an orbi-line
bundle $L^{-1}$ over $X$, and $h$ is a Hermitian metric on $L^{-1}$ such that
$\omega = \ddb\log h$ is positive on $X$. Letting $r$ be the fiberwise
norm, define the metric 
\[ \Omega = \ddb r^2 \]
on $Y\setminus\{0\}$. We will compute the Calabi functional of $\omega$ in
terms of the metric $\Omega$ on $Y$. 

Fix local coordinates $(z,w)$ where
$z\in X$ and $w$ is a local holomorphic section of $L^{-1}$ in a
neighborhood of $p = (z_{0},w_{0})$, and assume that $dh =0$ at $p$.  At
$p$ we compute
\begin{equation*}
\Omega = \ddb h(z)|w|^{2} = r^{2}\ddb \log h + h(z) \ddb |w|^{2} =
r^{2}\left(\pi^{*}\omega + \frac{i dw \wedge
d\overline{w}}{|w|^{2}}\right).
\end{equation*}
Here $\pi :Y \rightarrow X$ is the natural projection map.  It follows
that the Ricci form and scalar curvature of $\Omega$ are given by
\begin{equation*}
Ric({\Omega}) = (n+1)(\pi^{*}Ric(\omega)-\pi^{*}\omega),\qquad
R_{\Omega} =  r^{-2}(\pi^{*}R_{\omega}-n).
\end{equation*}
On a fixed fibre, the cylinder metric $|w|^{-2}(dw \wedge
d\overline{w})$ can also be written as $\frac{1}{r}dr \wedge d\theta$,
where $d\theta$ is given by the $U(1)$ action on the fibres of $L^{-1}$.
Hence, the volume form of $\Omega$ is
\begin{equation*}
\Omega^{n+1} = r^{2n+1}(\pi^{*}\omega)^{n} \wedge dr\wedge d\theta.
\end{equation*}
Let $\{U_{i}, \Gamma_{i}\}, i=1,\dots, n$ be a family of open sets
$U_{i} \subset \C^{n}$ together with local uniformizing groups
$\Gamma_{i}$, so that $U_{i}/\Gamma_{i}\cong V_{i} \subset X$ gives an
open cover of $X$, and so that $L^{-1}$ is trivial on each $V_{i}$.  Let
$\phi_{i}$ be a partition of unity subordinate to the cover $V_{i}$.
Note that the set $S:=\{r=1\} \subset Y$ is a smooth submanifold of $Y$,
which is the total space of a principal $U(1)$ orbibundle over $X$.
Thus, by Lemma 4.2.8 of \cite{BG}, we have that the local uniformizing
groups inject into $U(1)$.  In particular, the maps
\begin{equation*}
U(1) \times U_{i} \xrightarrow{\psi_{i}} V_{i}
\end{equation*}
are exactly  $|\Gamma_{i}|$-to-one on the complement of the orbifold
locus.  Let $\hat{R}_{\omega}$ denote that average scalar curvature of
$X$.  We compute
\begin{equation*}
\begin{aligned}
Cal_{X}(\omega)^{2}:=2\pi \int_{X} (R_{\omega} - \hat{R}_{\omega})^{2}
\omega^{n} &= \sum_{i} \frac{2\pi}{|\Gamma_{i}|} \int_{U_{i}} \phi_{i}
(R_{\omega} - \hat{R}_{\omega})^{2} \omega^{n}\\ &=
\sum_{i}\frac{1}{|\Gamma_{i}|}  \int_{U(1) \times U_{i}}
\pi^{*}\phi_{i}(\pi^{*}R_{\omega} - \hat{R}_{\omega})^{2}
\pi^{*}\omega^{n}\wedge d\theta \\&= \sum_{i=1}^{N} \int_{V_{i}}
\pi^{*}\phi(\pi^{*}R_{\omega} - \hat{R}_{\omega})^{2}
\pi^{*}\omega^{n}\wedge d\theta\\&= \int_{S}( \pi^{*}R_{\omega} -
\hat{R}_{\omega})^{2} \iota_{\frac{\del}{\del r}}(\Omega^{n+1}).
\end{aligned}
\end{equation*}
Let us write $\hat{R}_{\Omega}$ for the average of $R_{\Omega}$ when
restricted to $S$.  Then we have the relation
\begin{equation*}
\hat{R}_{\Omega} = \hat{R}_{\omega} - n.
\end{equation*}
Finally, we can compute
\begin{equation*}
\begin{aligned}
\int_{\{r\leq 1\}\subset Y} (r^{2}R_{\Omega} - \hat{R}_{\Omega})
\Omega^{n+1} &= \int_{0}^{1} \int_{S} ( \pi^{*}R_{\omega} -
\hat{R}_{\omega})^{2}\iota_{\frac{\del}{\del r}}(\Omega^{n+1})
r^{2n+1}dr \\ 
&= \frac{1}{2n+2}Cal_{X}(\omega)^{2}.
\end{aligned}
\end{equation*}
\begin{defn}
Let $Y$ be an affine variety with isolated singular point at $0$, and
Reeb field $\xi$ and let $\Omega$ be a K\"ahler metric on $Y$ compatible
with $\xi$, with scalar curvature $R_{\Omega}$.  Define
\begin{equation*}
Cal_{Y}(\Omega) := \left(\int_{\{r\leq 1\}} (r^{2}R_{\Omega} -
\hat{R}_{\Omega})^{2} \Omega^{n+1}\right)^{1/2}
\end{equation*}
where,
\begin{equation*}
\hat{R}_{\Omega} :=\frac{\int_{S}R_{\Omega}\iota_{\frac{\del}{\del
r}}(\Omega^{n+1})}{\int_{S}\iota_{\frac{\del}{\del r}}(\Omega^{n+1})}.
\end{equation*}
\end{defn}
In order to relate this to the Sasakian setting, let $(S,g)$ be a
Sasakian
manifold.  Observe that when $\xi$ is rational, 
\begin{equation}\label{scalar curvature relations}
\pi^{*}R_{\omega} = \frac{1}{4}R^{T} = \frac{1}{4}\left(R+2n\right),  
\end{equation}
where $R^{T}$ is the transverse scalar curvature of the Reeb foliation
and $R$ is the scalar curvature of the Sasakian metric $g$.  In this case,
we have
\begin{equation*}
\begin{aligned}
 \int_{S} ( \pi^{*}R_{\omega} -
 \hat{R}_{\omega})^{2}\iota_{\frac{\del}{\del r}}(\Omega^{n+1}) &=
 \frac{1}{16}\int_{S} (R^{T}-\hat{R}^{T})^{2}\iota_{\frac{\del}{\del
 r}}(\Omega^{n+1})\\ &= \frac{1}{16}\int_{S}(R - \hat{R})^{2}d\mu.
 \end{aligned}
 \end{equation*}
Here $\hat{R}$ is the average scalar curvature of $(S,g)$.  We note that
this agrees, up to a constant, with the functional studied in
\cite{BGS}.  We have shown that for rational Reeb fields we have the
equality
\begin{equation*}
Cal_{Y}(\Omega) = \frac{1}{4(2n+2)^{1/2}}Cal_{S}(g).
\end{equation*}
Since both sides of this equality depend continuously on the Reeb vector
field, this also holds for irrational Reeb fields by approximation.  We
record this in the following proposition.
\begin{prop}
Let $Y$ be an affine variety polarized by a rational Reeb field $\xi$.
Then,
\begin{equation*}
Cal_{Y}(\Omega) = \frac{1}{(2n+2)^{1/2}}Cal_{X}(\omega).
\end{equation*}
Moreover, when $Y$ is the cone over a Sasakian manifold $(S,g)$, then
\begin{equation*}
Cal_{Y}(\Omega) = \frac{1}{4(2n+2)^{1/2}}Cal_{S}(g),
\end{equation*}
and this holds for any Reeb vector field.
\end{prop}

Before proceeding, we make a few brief remarks about the scaling of the Calabi functional as a function of the Reeb field.  More precisely, suppose that $Y$ is an affine cone with Reeb vector field $\xi$, and a compatible K\"ahler metric $\Omega = i\del\dbar r^{2}$.  Scaling the Reeb vector field by a factor $\lambda >0$, corresponds to changing $r$  by $r \mapsto r^{\lambda}$.  This scaling yields a new metric $\Omega_{\lambda} = i\del\dbar r^{2\lambda}$.  It is straight forward to check that under a deformation of this type we have
\begin{equation}\label{Cal scale}
Cal_{Y}(\Omega_{\lambda}) = \lambda^{\frac{n-1}{2}}Cal_{Y}(\Omega).
\end{equation}

\section{The Index Character and the Donaldson-Futaki Invariant}
The main difficulty in extending the definition of K-stability to
irregular Sasakian manifolds, is 
the absence of a suitable Riemann-Roch formula when the Reeb
field is irrational.  When the Reeb field is rational, Ross-Thomas
showed in \cite{RT} that the relevant coefficients are the non-periodic
terms of the orbifold Riemann-Roch expansion (see Section~\ref{sec:orb}).  
When the Reeb vector
field is irrational, the Riemann-Roch expansion does not seem to exist,
so we would like to define the relevant coefficients by approximating an
irrational Reeb vector field $\xi$ by a sequence of rational ones
$\xi_k$. The periodic terms in the expansions \eqref{eq:expand}
corresponding to the $\xi_k$ become
unmanageable as $k\to\infty$, so we need a different approach. 
 The key observation is that
the Riemann-Roch coefficients are determined by the Hilbert series, or
equivalently the index character introduced by
Martelli-Sparks-Yau~\cite{MSY}. For
the leading term (the volume), this was also used by \cite{MSY}. 

In this section and the next we take $Y \subset \mathbb{C}^{N}$ to be an
affine scheme of dimension $n+1$, defined by the ideal $I = (f_{1},
\dots, f_{r}) \subset R = \mathbb{C}[x_{1}, \dots, x_{N}]$.   Let $T
\subset GL(N,\C)$ be a torus of dimension $s$ acting diagonally,
holomorphically and effectively on $Y$.  We make this assumption without
loss of generality by Lemma~\ref{lin act}.    Denote by $\mathfrak{t}$
the Lie algebra of $T$ and let $\mathcal{H} = R/I$ be the ring of
regular functions on $Y$.  Since $T$ fixes $Y$, the ideal $I$ is
homogeneous for the torus action. By Corollary~\ref{approx cor} 
we may always assume that $T$ contains at least one rational Reeb
vector field.  Let
\begin{equation*}
\mathcal{H} = \oplus_{\alpha \in \mathfrak{t}^{*}}  \mathcal{H}_{\alpha}
\end{equation*}
be the weight space decomposition of $\mathcal{H}$.

\begin{defn}
In the above situation, we define the $T$-equivariant index character
$F(\xi,t)$ for $\xi\in\mathcal{C}_{R}$ and $t\in \mathbb{C}$ with
$\mathrm{Re}(t) > 0$, by
\begin{equation}\label{indchar}
F(\xi,t) := \sum_{\alpha \in \mt^{*}} e^{-t\alpha(\xi)}\dim\mathcal{H}_{\alpha}.
\end{equation}
\end{defn}

\begin{lem}\label{ind char convg}
The defining sum for $F(\xi, t)$ converges if $\xi$ is a Reeb vector
field and $\mathrm{Re}(t) > 0$.  
\end{lem}
\begin{proof}
The dimensions $\dim\mathcal{H}_\alpha$ are bounded by the corresponding
dimensions for $\C^N$.  As $\xi$ acts by positive weights,
$\dim\mathcal{H}_\alpha < C|\alpha|^N$.  Moreover, since $\xi$ is a Reeb
vector field, there is a $c > 0$ such that
$\alpha(\xi) >c|\alpha|$ for all $\alpha$ with non-zero
$\mathcal{H}_\alpha$.  We obtain
\begin{equation*}
	\sum_{\alpha \in \mathfrak{t}^{*}}\left|
	e^{-t\alpha(\xi)}\right|\dim\mathcal{H}_{\alpha} \leq
	C\sum_{\alpha\in\mathfrak{t}^*} e^{-c|\alpha|\mathrm{Re}(t)}|\alpha|^N,
\end{equation*}
which converges if $\mathrm{Re}(t)>0$.
\end{proof}

Suppose that $\xi$ is rational, and it is minimal satisfying the
condition that $\alpha(\xi)$ is integral for each $\alpha$ with
non-zero weight space.
Then as before we can think of $Y$ as the total space of a line bundle
$L$ over the orbifold $X=Y/\C^*$, and 
\begin{equation*}
H^0(X,L^k) = \bigoplus_{\alpha;\,\alpha(\xi)=k} \mathcal{H}_\alpha.
\end{equation*}
By the orbifold Riemann-Roch theorem \cite{Kaw, Toe}, we have
\begin{equation*}
\dim H^0(X,L^k) = a_0k^n + (a_1 + \rho)k^{n-1} + \cdots
\end{equation*}
for some periodic function $\rho$ with average zero. In this case we
have the following.
\begin{prop}\label{ind char RR}
The $T$-equivariant index character $F(\xi,t)$ as a function of $t$ has
a meromorphic extension to a neighborhood of the origin, and it has
Laurent expansion 
\begin{equation*}
F(\xi,t) = \frac{a_{0}n!}{t^{n+1}} + \frac{a_{1}(n-1)!}{t^{n}}+ O(t^{1-n}),
\end{equation*}
near $t=0$. 
\end{prop}
\begin{proof}
	By definition we have
\begin{equation*}
\begin{aligned}
F(\xi,t) &= \sum_{k=0}^\infty e^{-kt}\,\dim H^0(L^k) \\
		&= \sum_{k=0}^\infty e^{-kt}\left( a_0k^n +
		(a_1 + \rho)k^{n-1}+O(k^{n-2})\right).
\end{aligned}
\end{equation*}
Note that 
 \begin{equation*}
 \sum_k e^{-tk} = \frac{1}{1-e^{-t}} = \frac{1}{t} + f(t),
 \end{equation*}
 where $f$ is analytic, so differentiating $n$ times with respect to $t$, we get
\begin{equation*}
\sum_k e^{-tk} k^n = \frac{n!}{t^{n+1}} + (-1)^n f^{(n)}(t).
\end{equation*}
Moreover, $ G(t) = \sum_k \rho(k)e^{-tk}$ is analytic near $t=0$ since
$\rho$ has average zero. Indeed if $d$ is the period of $\rho$ then we
have 
\begin{equation*}
	\sum_k (\rho(k)+\rho(k+1)+\ldots+\rho(k+d-1))e^{-tk} = 0,
	\end{equation*}
and so
\begin{equation*}
	G(t) + e^t(G(t)-\rho(0)) + \ldots + e^{(d-1)t} \left( G(t)
	-\sum_{k=0}^{d-2}\rho(k)e^{-kt}\right) = 0 ,
\end{equation*}
and therefore 
\begin{equation*}
G(t) = \frac{H(t)}{1+e^t+e^{2t}+\ldots+e^{(d-1)t}}
\end{equation*}
where $H(t)$ is analytic since it is a finite sum. It follows that
$G(t)$ is also analytic near $0$, with poles at $t=\frac{2\pi i k}{d}$
for non-zero integers $k\not=0$. Finally it follows that $F(\xi,t)$ is
meromorphic near $t=0$ with a pole at the origin, and we
have
\begin{equation*}
	F(\xi,t) = \frac{a_0n!}{t^{n+1}} + \frac{a_1(n-1)!}{t^n} + O(t^{1-n}).
\end{equation*}
\end{proof}

In particular, we can read off the Riemann-Roch coefficients from the
index character. 
The main advantage of this observation is that the index
character is defined even when the Reeb vector field is irrational, and
we can hope to use the asymptotics at $t=0$ to extract the coefficients
$a_0, a_1$ which are needed for the definition of the Futaki invariant.
In addition we will see that these coefficients vary
smoothly as we vary $\xi$, so we will be able to use an approximation
argument to prove our main result, Theorem~\ref{Don thm}. The difficutly
of working more directly with the Riemann-Roch expansions on the
quotient orbifolds is that as we
vary the Reeb field, the orbifolds change, and the periodic terms in the
expansions can become more and more unmanageable. The index character,
on the other hand encodes the relevant coefficients in the Riemann-Roch
expansions for all quotient orbifolds at the same time, so it becomes
possible to study their variation as we vary $\xi$. In addition the
index character is essentially the multivariate Hilbert series of a
multigraded module, and thus it can be readily computed in examples.

The main observation is that fixing a Reeb field $\xi$ gives rise to
a grading on
$R=\C[x_{1},\dots,x_{N}]$, and on $R/I$, where $I$ cuts out the variety
$Y$. Writing $HS_{R/I}(t)$ for the Hilbert series of the graded ring $R/I$,
we have
\begin{equation*}
F(\xi,t) = HS_{R/I}(e^{-t}).
\end{equation*}

We will now assume that the embedding $Y\subset\mathbb{C}^N$ is obtained
through an application of Lemma~\ref{lin act}. The corresponding ideal
$I\subset R$ is then homogeneous with respect
to a multigrading on $R$.

More precisely, let $E := \{e^{*}_{1}, \dots,
e^{*}_{s}\}$ be an integral basis of $\mathbb{R}^{s} \cong \mathfrak{t}^*$ 
and let
$\alpha_{i}$ be the weight of the representation on the generator
$x_{i}$ of $R$.  Expressing the $\alpha_{i}$ in the basis $E$ yields an $s
\times N$ matrix
\begin{equation}
W =\begin{pmatrix}
\alpha_{1,1}&\alpha_{1,2}&\dots&\alpha_{1,N}\\
\vdots & \vdots&\ddots&\vdots \\
\alpha_{s,1}& \alpha_{s,2}&\dots&\alpha_{s,N}
\end{pmatrix}
\end{equation}
with integer entries.
Since, $R$ is graded by $W$, and $I$ is homogenous, it
follows that $R/I$ is a $W$-graded $R$ module, generated in degree zero.

\begin{defn}
Let $s\geq 1$, let $R$ be graded by a matrix $W$ of rank $s$ in
$Mat_{s,N}(\mathbb{Z})$, and let $\alpha_{1}, \dots, \alpha_{s}$ be the rows of
$W$.  The grading on $R$ given by $W$ is of \emph{positive type} if
there exists $a_{1},\dots, a_{s} \in \mathbb{Z}$ such that all the
entries of $a_{1}\alpha_{1} + \dots + a_{s}\alpha_{s}$ are positive.
\end{defn}

\begin{lem} If there exists a Reeb vector field in $\mathfrak{t}$, then the
	grading induced by $W$ is of positive type.
\end{lem} 
\begin{proof} We first need to show that $W$ has rank $s$.
	Observe that if $v^{T}\cdot W =0$, then the action induced by
	$v$ is trivial.  In particular, the action of $T$ is not
	effective.  Secondly, by Corollary~\ref{approx cor} 
	we can assume that there is an integral Reeb field
	$\xi\in\mathfrak{t}$, given in terms of the dual basis
	$\{e_1,\ldots,e_s\}$ by a vector $(a_1,\ldots,a_s)$ with $a_i\in
	\mathbb{Z}$. The entries of $a_1\alpha_1 + \ldots a_s\alpha_s$
	are the weights of the action induced by $\xi$ on the generators
	$x_1,\ldots,x_N$. By definition of a Reeb field these are all
	positive. 
\end{proof}

We will now recall some results about multigradings and the multigraded
Hilbert function. 

\begin{lem}[\cite{KR}, Proposition 4.1.19] Let $R = \mathbb{C}[x_{1},
	\dots, x_{N}]$ be graded by a matrix $W \in
	Mat_{m,N}(\mathbb{Z})$ of positive type, and let $M$ be a
	finitely generated graded $R$-module. Then, \begin{enumerate}
	\item $R_{0} = \C$.  That is, the degree zero elements
			in $R$ are precisely the constants.
	\item For every $d \in \mathbb{Z}^{m}$, we have
		$\dim_{\mathbb{C}}(M_{d}) <\infty$
\end{enumerate}
\end{lem}

The previous lemma indicates that the following definition makes sense;

\begin{defn}[\cite{KR}, Definition 5.8.8, 5.8.11]
Let $R$ be graded by a matrix $W \in Mat_{m,N}(\mathbb{Z})$, and let $M$
be a finitely generated, graded $R$ module.  Then the map
\[ \begin{aligned}
	HF_{M,W}:\mathbb{Z}^{m} &\rightarrow \mathbb{Z} \\
	(i_{1},\dots,i_{m}) &\mapsto \dim_{\mathbb{C}}(M_{i_{1},\dots, i_{m}})
\end{aligned}\]
for all $(i_{1},\dots, i_{m}) \in \mathbb{Z}^{m}$ is called the
\emph{multigraded Hilbert function} of M with respect to the grading
$W$.  We may define the multivariate power Hilbert series of $M$ with
respect to the grading $W$ by
\begin{equation}
HS_{M,W}(z_{1},\dots z_{m}) = \sum_{(i_{1},\dots, i_{m})\in
\mathbb{Z}^{m}} HF_{M,W}(i_{1},\dots, i_{m})z_{1}^{i_{1}}\cdots
z_{m}^{i_{m}} \in \mathbb{Z}[[{\mathbf z}, {\mathbf z}^{-1}]]
\end{equation}
\end{defn} 

The following lemma provides a convenient characterization of
multivariate Hilbert series under changes in the grading.

\begin{lem}[\cite{KR}, Proposition 5.8.24]\label{HS grading change}
Let $W \in Mat_{m,N}(\mathbb{Z})$, and $A = (a_{ij}) \in
Mat_{l,m}(\mathbb{Z})$ be two matrices such that the gradings on $R =
\mathbb{C}[z_{0},\dots, z_{N}]$ given by $W$ and $A\cdot W$ are both of
positive type.  Let $M$ be a finitely generated $R$-module which is
graded with respect to the grading given by $W$.  Then the Hilbert
series of $M$ with respect to the grading given by $A\cdot W$ is given
by
\begin{equation*}
HS_{M, A\cdot W}(z_{1},\dots,z_{l}) = HS_{M,W}(z_{1}^{a_{11}}\cdots
z_{l}^{a_{l1}}, \dots, z_{1}^{a_{1m}}\cdots z_{l}^{a_{lm}})
\end{equation*}
\end{lem}

If R is graded by $W = (w_{ij}) \in Mat_{m,n}(\mathbb{Z})$ of positive
type, and $\xi$ is a Reeb field, then the grading induced by $\xi^{T}
\cdot W$ is clearly of positive type, and so the above lemma describes
the relation between the multigraded Hilbert series and the index
character.  The next proposition describes the general shape of
multivariable Hilbert series.

\begin{prop}[\cite{KR}, Corollary 5.8.19]\label{HS shape prop}
Let $R$ be graded by $W \in Mat_{m,N}(\mathbb{Z}) $, a matrix of
positive type.  Let $M$ be a finitely generated, graded $R$-module, and
$(m_{1},\dots,m_{r})$ be a tuple of non-zero homogeneous elements of $M$
which form a minimal system of generators.  For $i=1,\dots, r$, let
$d_{i} = deg_{W}(m_{i})$.  Then  the multivariate Hilbert series of $M$
has the following form;
\begin{equation*}
HS_{M,W}(z_{1},\dots z_{m}) = \frac{z_{1}^{\alpha_{1}}\cdots
z_{m}^{\alpha_{m}} \cdot
HN(z_{1},\dots,z_{m})}{\prod_{j=1}^{N}(1-z_{1}^{w_{1j}}\cdots
z_{m}^{w_{mj}})}
\end{equation*}
where $(\alpha_{1},\dots,\alpha_{m})$ is the component wise minimum of
$\{d_{i}\}$, and $HN_{M}(z_{1},\dots,z_{m})$ is a polynomial in
$\mathbb{Z}[z_{1},\dots,z_{m}]$.
\end{prop}

We can now translate this result to the language of index characters. 

\begin{thm}\label{ind char exp}
Let $Y \subset \C^{N}$ be an affine scheme of dimension $n+1$, and
suppose that $T \subset GL(N,\C)$ is a torus acting effectively,
diagonally and holomorphically on $Y$.  Let $\mt$ be the Lie algebra of
$T$, and $\mathcal{C}_{R} \subset \mt$ be the Reeb cone. For fixed
$\xi\in \mathcal{C}_R$ the index character
$F(\xi, t)$ has a meromorphic extension to $\mathbb{C}$ 
with poles along the imaginary axis. Near $t=0$ it has a Laurent series
\begin{equation}
F(\xi, t) = \frac{a_{0}(\xi)n!}{t^{n+1}} + \frac{a_{1}(\xi)(n-1)!}{t^{n}} +
\dots,
\end{equation}
where $a_{i}(\xi)$ depend smoothly on $\xi \in \mathcal{C}_{R}$, and $a_{0}(\xi)>0$.  
\end{thm}
\begin{proof}
	As above, with a basis of $\mathfrak{t}\cong\mathbb{R}^{s}$
	fixed, write $\xi
= (\xi_{1}, \dots,\xi_{s})$ for an element of $\mt$.  By
Proposition \ref{HS shape prop}, the Hilbert series of the grading
induced by $W$ is given by
\begin{equation*}
HS_{W}(e^{-t_{1}}, \dots, e^{-t_{s}}) =\frac{e^{-t_{1}\alpha_{1}}\cdots
e^{-t_{s}\alpha_{s}} \cdot
HN(e^{-t_{1}},\dots,e^{-t_{s}})}{\prod_{j=1}^{N}(1-e^{-t_{1}w_{1j}}\cdots
e^{-t_{s}w_{sj}})}
\end{equation*}
where $\alpha_{i} \geq 0$ for every $i$.  By Lemma \ref{HS grading change}, we obtain
\begin{equation*}
F(\xi,t) = HS_{\xi^{T}\cdot W}(e^{-t}) =
\frac{e^{-t(\xi_{1}\alpha_{1}+\cdots +\xi_{s}\alpha_{s})} \cdot
HN(e^{-t\xi_{1}},\dots,e^{-t\xi_{s}})}{\prod_{j=1}^{N}
(1-e^{-t(\xi_{1}w_{1j}+\cdots+\xi_{s}w_{sj})})}
\end{equation*}
From this formula it follows that $F(\xi,t)$ is a meromorphic function
with coefficients depending smoothly on the Reeb field. More precisely 
for fixed $\xi\in\mathcal{C}_R$ there are no poles other than the origin
in the ball where
\[ |t| < \frac{2\pi}{\max_j\{ \xi_1w_{1j} + \ldots + \xi_s w_{sj}\}},\]
so we can compute the coefficients of the Laurent series using the
Cauchy integral formula on a small circle around the origin. As long as
$\xi$ varies in a bounded subset of $\mathcal{C}_R$, we can use the same
circle around the origin, and the coefficients will vary smoothly with
$\xi$. It
follows also that the order of the pole at $t=0$ is determined by the
order of the pole for rational $\xi$, which is $n+1$ by
Proposition~\ref{ind char RR}. 
Note that the coefficients blow up at the boundary of the
Reeb cone, since as $\xi$ approaches the boundary, there will be a $j$
such that $\xi_1w_{1j} + \ldots + \xi_s w_{sj}\to 0$.
\end{proof}

In some special cases, we can recover this result by computing the index
character explicitly.  For example, we have

\begin{prop}\label{ci}
Let $Y$ be a complete intersection, determined by the regular sequence
$f_{1}=\dots=f_{k}=0$.  Let $\alpha_{i}$ be the weight of the generators
$x_{i}$, and let $\beta_{j}$ be the weight of $f_{j}$.  Then we have
\begin{equation*}
F(\xi,t) = \frac{\prod_{j=1}^{k}
(1-e^{-t\beta_{j}(\xi)})}{\prod_{i=0}^{N}(1-e^{-t\alpha_{i}(\xi)})}
\end{equation*}
\end{prop}
\begin{proof}
Use the degree shifted Koszul complex resolution of $R/I$, and compute
the Hilbert series.
\end{proof}

In order to define the Futaki invariant, we need equivariant versions
of the index character, taking into account an extra
$\mathbb{C}^*$-action. 
\begin{defn}
	In the situation of Theorem~\ref{ind char exp} with
	$\xi\in\mathcal{C}_R$, suppose $\eta \in \mt$,
and define the weight characters 
\begin{equation*}
	\begin{aligned}
C_\eta(\xi,t) &=
\sum_{\alpha\in\mathfrak{t}^*}e^{-t\alpha(\xi)}\alpha(\eta),\\
C_{\eta^2}(\xi,t) &= \sum_{\alpha\in\mathfrak{t}^*} e^{-t\alpha(\xi)}
(\alpha(\eta))^2.
\end{aligned}
\end{equation*}
\end{defn}
The convergence of these weight character follows from the arguments in
Lemma~\ref{ind char convg}.   As before, when $\xi$ is rational, 
we obtain a line bundle $L$ over the orbifold
$X=Y/\C^*$ with a $\C^*$-action on $L$ generated by
$\eta$.  Adapting the computations preceeding Proposition~\ref{ind char
RR} proves the following. 
\begin{prop}\label{weight char RR}
In the situation of Theorem~\ref{ind char exp}, with $\xi$ rational,
write $A_k$ for the infinitesimal action of $\eta$ on
$H^0(X, L^k)$, and define $b_0,b_1,c_0$ by the expansions
\begin{equation*}
	\begin{aligned}
\mathrm{Tr}(A_k) &= b_0k^{n+1} + (b_1 + \rho)
k^{n} + O(k^{n-1}),\\
\mathrm{Tr}(A_k^2) &= c_0k^{n+2} + O(k^{n+1}),
\end{aligned}
\end{equation*}
where $\rho$ is a periodic function with average zero, and $c_{0} \geq 0$.  Then the weight
characters have the asymptotic expansions
\begin{equation*}
	\begin{aligned}
	C_{\eta}(\xi,t ) &= \frac{b_0(n+1)!}{t^{n+2}} + \frac{b_1n!}{t^{n+1}}+
	O(t^{-n}),\\
	C_{\eta^2}(\xi,t) &= \frac{c_0(n+2)!}{t^{n+3}} + O(t^{-n-2}). 
\end{aligned}
\end{equation*}
\end{prop}

We remark that the inequality $c_{0} \geq 0$ follows from equation (2.20) in \cite{RT}, orThe results of Theorem~\ref{ind char exp} can also be extended quite
easily. 

\begin{thm}\label{equivariant RR}
In the situation of Theorem~\ref{ind char exp}, with $\eta \in
\mathfrak{t}$, the weight characters
admit meromorphic expansions to a small neighbourhood of $0 \in
\mathbb{C}$ of the form
\begin{equation*}
	\begin{aligned}
C_{\eta}(\xi,t) &= \frac{b_{0}(\xi)(n+1)!}{t^{n+2}} +
\frac{b_{1}(\xi)n!}{t^{n+1}}+O(t^{-n}) \\
C_{\eta^2}(\xi,t) &= \frac{c_0(\xi)(n+2)!}{t^{n+3}} + O(t^{-n-2}),
\end{aligned}
\end{equation*}
where $b_{0}, b_{1}, c_0$ depend smoothly on $\xi \in \mathcal{C}_{R}$.
Moreover, we have
\begin{equation*}
	\begin{aligned}
b_{i}(\xi) &= \frac{-1}{(n+1-i)}D_{\eta}a_{i}(\xi) \quad \text{ for } i=0,1 \\
c_0(\xi) &=\frac{1}{(n+2)(n+1)} D_{\eta}^2a_0(\xi),
\end{aligned}
\end{equation*}
where $D_{\eta}$ denotes the directional derivative along $\eta$ in
$\mathbb{R}^{s} \cong \mt$.
\end{thm}
\begin{proof}
We define
\begin{equation*}
G(\xi,s, t) = \sum_{\alpha \in \mathfrak{t}^{*}}e^{-t\alpha(\xi-s\eta)}\dim H_{\alpha}.
\end{equation*}
For $s$ sufficiently small, $\xi-s\eta$ is a Reeb vector field and so
the defining sum for  $G(\xi,s,t)$ converges uniformly for $t>0$, and we
have $G(\xi,s,t) = F(\xi-s\eta, t)$.  It is clear that
\begin{equation*}
tC_{\eta}(\xi,t) = \frac{\partial}{\partial s} G(\xi,s,t)\bigg|_{s=0}=
\frac{\partial}{\partial s} \left(\frac{a_{0}(\xi-s\eta)n!}{t^{n+1}} +
\frac{a_{1}(\xi-s\eta)(n-1)!}{t^{n}}+\cdots \right)\bigg|_{s=0}.
\end{equation*} 
By Theorem~\ref{ind char exp} the coefficients $a_{0}, a_{1}, \dots$
depend smoothly on the Reeb field and so we can differentiate term by
term to obtain
\begin{equation*}
C_{\eta}(\xi,t) = \frac{b_{0}(\xi)(n+1)!}{t^{n+2}} + \frac{b_{1}(\xi)n!}{t^{n+1}}+\cdots,
\end{equation*}
where, for example,  $b_{0}(\xi) = \frac{-1}{n+1}D_{\eta}a_{0}(\xi)$ and $D_{\eta}$
denotes the directional derivative along $\eta$. The argument for
$C_{\eta^2}$ is identical.
\end{proof}

\section{Test Configurations for Polarized Affine Varieties}
Our first task is to define a test configuration for an affine variety
$Y$ polarized by a Reeb field $\xi$. Recall that we can assume that
$Y\subset\mathbb{C}^N$ is invariant under the linear action of a torus $T$
and the Reeb field $\xi$ is in the Lie algebra $\mt$ of the maximal compact
subtorus. Let $\mathcal{H}$ be the coordinate ring of $Y$. 
\begin{defn}\label{test config}
	A $T$-equivariant test-configuration for $Y$ consists of the
	following data. 
	\begin{enumerate}
		\item A set of $T$-homogeneous elements
			$\{f_1,\ldots,f_k\}\in\mathcal{H}$, which
			generate $\mathcal{H}$ in sufficiently high
			degrees. 
		\item Integers $w_i$ for $i=1,\ldots,k$. 
	\end{enumerate}
\end{defn}

This corresponds to the usual, more geometric definition of
test-configu\-ra\-tions. Namely we can embed $Y$ into $\mathbb{C}^k$ using the
functions $\{f_1,\ldots,f_k\}$, and then act on $\mathbb{C}^k$ by the
$\mathbb{C}^*$-action with weights $w_i$. Taking the flat limit across $0$
of the $\mathbb{C}^*$-orbit of $Y$ we obtain a flat family of affine
schemes over $\mathbb{C}$. It is in this form that we will construct our
test configurations in Section~\ref{sec:appl}.
The central fiber $Y_0$ still has an action of
$T$ as well as a new $\mathbb{C}^*$-action commuting with $T$ (if we have a
product configuration, then this new $\mathbb{C}^*$ is actually a subgroup
of $T$).  Note that 
when $\xi$ is rational, then we can take $T$ to be the
1-dimensional torus generated by $\xi$, and a test-configuration for $Y$ is
the same as a test-configuration for the quotient orbifold as we defined it
in Section~\ref{sec:orb}.

It is important to note that as a $T$-representation, the ring of functions
on the central fiber
$Y_0$ is isomorphic to $\mathcal{H}$, it is only the multiplicative
structure that changes. In particular if $\xi\in\mt$ is a Reeb field on $Y$,
then it is also a Reeb field on $Y_0$. We can therefore apply our results
on the index character to $Y_0$. 
By Theorem~\ref{ind char exp}, the index character expands
asymptotically as
\begin{equation*}
F(\xi, t) = \frac{a_{0}(\xi)n!}{t^{n+1}} + \frac{a_{1}(\xi)(n-1)!}{t^{n}} + O(t^{1-n})
\end{equation*}
where $a_{0}, a_{1} :\mathcal{C}_{R} \rightarrow \mathbb{R}$ are smooth
functions.  Moreover, $Y$ inherits an extra $\C^{*}$ action generated by
$\eta \in \mt'=\text{Lie}(T'_\mathbb{R})$ for some torus 
$T' \subset GL(N,\C)$ with
$T \subset T'$.  By Theorem~\ref{equivariant RR}, the weight characters
expand as 
\begin{equation*}
\begin{aligned}
C_{\eta}(\xi, t) &= \frac{b_{0}(\xi)(n+1)!}{t^{n+2}} +
\frac{b_{1}(\xi)n!}{t^{n+1}} + O(t^{-n})\\
C_{\eta^{2}}(\xi, t) &= \frac{c_{0}(\xi)(n+2)!}{t^{n+3}} + O(t^{2-n})
\end{aligned}
\end{equation*}
where $b_{0}, b_{1}, c_{0} :\mathcal{C}_{R} \rightarrow \mathbb{R}$ are
smooth functions, and $c_{0}\geq0$.  

\begin{defn}
In the above situation, we define the Donaldson-Futaki invariant of the
test configuration, with respect to the Reeb field $\xi$, by 
\begin{equation}\label{DF invariant}
	Fut(Y_{0},\xi,\eta) := \frac{a_{1}(\xi)}{a_0(\xi)}b_{0}(\xi) -
b_{1}(\xi) =
\frac{a_0(\xi)}{n}D_{\eta}(\frac{a_{1}}{a_{0}})(\xi) +
\frac{a_{1}(\xi)D_{\eta}a_{0}(\xi)}{n(n+1)a_{0}(\xi)},
\end{equation}
where the second equality follows from Theorem~\ref{equivariant RR}. We
also define the norm of $\eta$, with respect to the Reeb field $\xi$ by
\[ \Vert\eta\Vert^2_\xi = c_0(\xi) - \frac{b_0(\xi)^2}{a_0(\xi)}. \] 
\end{defn}

Propositions~\ref{ind char RR} and~\ref{weight char RR} show that the
above definition of the Donaldson-Futaki invariant extends Ross-Thomas's
orbifold Donaldson-Futaki invariant to irrational Reeb vector fields.

\begin{defn}
We say that $(Y,\xi)$ is K-semistable if, for every torus $T \ni \xi$,
and every $T$-equivariant  test configuration with central fibre $Y_{0}$, we have
\begin{equation*}
Fut(Y_{0},\xi, \eta) \geq 0
\end{equation*}
where $\eta \in T'$ is the induced $\C^{*}$ action on the central fibre.  
\end{defn}

K-stability could also be defined along similar lines. Since there is
usually a positive dimensional torus of automorphisms, one natural way
would be to use the notion of relative stability following \cite{Sz1}. This 
would also
allow us to consider the analogs of extremal metrics (called canonical
Sasakian metrics in \cite{BGS}). Since we do not use these notions in this
paper, we will not define them. 
We are now in a position to prove our main theorem;

\begin{thm}\label{Don thm}
Let $(Y,\xi)$ be a polarized affine variety of dimension $n+1$ with a torus
of automorphisms $T$, containing the Reeb field. Suppose
that we have a $T$-equivariant test-configuration for $Y$ and let
$Y_0$ be the central fibre with induced $\C^{*}$-action $\eta$.  
For any K\"ahler metric $\Omega$ on $Y$ compatible with $\xi$,
\begin{equation}\label{Cal bound}
\Vert\eta\Vert_\xi\cdot Cal_{Y}(\Omega) \geq -c(n)(\xi) Fut(Y_{0},\xi,\eta),
\end{equation}
where $c(n)$ is a strictly positive constant depending only on $n$.
\end{thm}

\begin{proof}
When $\xi$ is rational and minimal satisfying the condition that
$\alpha(\xi)$ is integral for each $\alpha \in \mt^{*}$ with non-empty
weight space,  this theorem is just a restatement of the results of
Donaldson \cite{D} and Ross-Thomas \cite{RT}.  However, from the
definitions of $a_{i}, b_{i}, c_{0}$ for $i=0,1$, and the scaling of the
Calabi functional in equation~(\ref{Cal scale}), the inequality is
invariant under scaling the Reeb field.  In particular, it holds for all
rational Reeb vector fields.  

Assume that $\xi$ is irrational.  According to Corollary~\ref{approx cor} we can
approximate $\xi$ with a sequence of rational Reeb fields $\xi_k\in\mt$,
and find corresponding compatible K\"ahler metrics $\Omega_k$, which
converge to $\Omega$ smoothly on compact sets. For the rational $\xi_k$ we
already know that 
\begin{equation*}
	\Vert\eta\Vert_{\xi_k}\cdot Cal_{Y}(\Omega_{k}) \geq -c(n)
Fut(Y_{0},\xi_{k}, \eta).
\end{equation*}
All the terms in this inequality depend smoothly on the Reeb vector field
by Theorems~\ref{ind char exp}
and~\ref{equivariant RR}, and since $\Omega_k\to\Omega$ smoothly on compact
sets. We can therefore take a limit as $k\to\infty$ to obtain the
inequality for the irrational Reeb field $\xi$. 
\end{proof}

Corollary~\ref{main cor}, stated in the introduction,
follows immediately from Theorem~\ref{Don thm}, since $c(n) > 0$.

\section{Applications and Examples}\label{sec:appl}

As an application of our techniques, we will show that the the volume
minimization results of \cite{MSY} and the Lichnerowicz obstruction of
\cite{GMSY} can be obtained directly from K-stability considerations as
obstructions to existence of Sasaki-Einstein metrics.  More precisely,
we will show that for Calabi-Yau cones with isolated Gorenstein
singularities, and a torus action inducing a Reeb vector field,
K-stability for product test configurations implies the volume
minimization results of \cite{MSY}.  Martelli, Sparks and Yau noticed
that when the Reeb field minimizing the volume functional was rational,
the Futaki invariant on the quotient orbifold vanished.  Secondly,  we
will apply the Rees deformation to interpret the Lichnerowicz
obstruction of \cite{GMSY} in terms of K-stability. For rational Reeb
vector fields, the Lichnerowicz obstruction was shown to imply the slope
instability, and hence K-instability, of the quotient orbifold in
\cite{RT}.  Our results recover this theorem, and extend it to the
setting of irrational Reeb fields.

Let $Y$ be an affine, Calabi-Yau variety with an isolated singularity at
$0$, and a torus $T$, acting holomorphically, and effectively on $Y$,
admitting a Reeb vector field $\xi \in \mt$.  We suppose that $0 \in Y$
is a Gorenstein singularity, by which we mean that the canonical bundle
is trivial on $X := Y-\{0\}$.  According to section 2.7 of \cite{MSY},
we fix a non-vanishing section $\Theta \in H^{0}(X, K_{X})$ which is
homogeneous of degree $n+1$ for the action of the Reeb field.  More
precisely, we fix a cross-section $\Sigma \subset \mathcal{C}_{R}$ so
that for each $\xi \in \Sigma$, we have $\mathcal{L}_{\xi}\Theta =
i(n+1)\Theta$.  According to \cite{MSY}, $\Sigma$ is a compact, convex
polytope.  By the computations in Section 3.1 of \cite{MSY}
\begin{equation*}
\int_{S}R(g_{S})d\mu =2n(2n+1)Vol(S).
\end{equation*}
Assuming for the moment that $\xi \in \Sigma$ is rational, the orbifold
Riemann-Roch theorem implies that
\begin{equation*}
a_{1}(\xi) = \frac{1}{2}\int_{X}R_{\omega} \frac{\omega^{n}}{n!} = \frac{1}{16\pi
}\int_{S} R(g_{S})d\mu +\frac{2n}{16\pi }Vol(S).
\end{equation*}
This follows from a computation similar to the computation in section 4
for the Calabi functional, and the relation between the \emph{complex}
transverse scalar curvature of the Reeb foliation and the \emph{real}
scalar curvature of the Sasakian metric given by equation~(\ref{scalar
curvature relations}).  Moreover, we have
\begin{equation*}
a_{0}(\xi) = \frac{1}{2\pi}Vol(S),
\end{equation*}
which follows easily by a similar argument.  Since both of these
identities are continuous in the Reeb field, they extend from the
rational Reeb fields to all of $\Sigma$.  Thus, for $\xi \in \Sigma$ we
have
\begin{equation}\label{Gore relation}
a_{1}(\xi) = \frac{n(n+1)}{2}a_{0}(\xi).
\end{equation}
Consider now a product test configuration $Y \times \C$, with a $\C^{*}$
action generated by $\eta \in \mt$.  We assume additionally that  $\eta$
is tangent $\Sigma$.  Applying equation~(\ref{DF invariant}), the
Donaldson-Futaki invariant is given by
\begin{equation*}
 Fut(Y,\xi,\eta) = \frac{1}{2}D_{\eta}a_{0}(\xi).
 \end{equation*}
Since we could replace $\eta$ with $-\eta$, 
it follows from Theorem~\ref{Don thm}
that If $\xi$ is the Reeb vector field of a Sasaki-Einstein
metric, then we must have
\begin{equation*}
D_{\eta}a_{0}(\xi)=0
\end{equation*}
for every rational $\eta$, and hence $\xi$ must be a critical point of the volume
functional. 
Moreover, it was shown in \cite{MSY} that the volume functional of a
Sasakian
manifold is strictly convex when restricted to $\Sigma$, so a critical
point is necessarily a minimum. 
In particular, we have

\begin{thm}\label{V thm}
Let $(Y,\Theta)$ be an isolated Gorenstein singularity with link $L$,
and Reeb vector field $\xi$ satisfying $\mathcal{L}_{\xi}\Theta=
i(n+1)\Theta$.  If $\xi$ does not minimize the volume functional of the
link $L$, then $(Y,\xi)$ is K-unstable.
\end{thm}

From the argument it is clear that more generally
in any family of Reeb fields
$\xi$ for which the ratio $a_1/a_0$ is constant, a K-semistable Reeb field
must be a critical point of the volume $a_0$. 
In the case of Gorenstein singularities we 
obtain the following corollary, which was first pointed out in \cite{MSY}.

\begin{cor}
Let $(Y,\Theta)$ be an isolated Gorenstein singularity with link $L$,
and Reeb vector field $\xi$ satisfying $\mathcal{L}_{\xi}\Theta=
i(n+1)\Theta$.  If $\xi$ does not minimize the volume functional of the
link $L$, then $(Y,\xi)$ does not admit a compatible K\"ahler metric
with constant scalar curvature.  In particular, the link $L$ with Reeb
field $\xi$ does not admit a Sasaki-Einstein metric.
\end{cor}

Next, we aim to show how the Lichnerowicz obstruction of Gauntlett, Martelli, Sparks and Yau \cite{GMSY} can
be interpreted in terms of K-stability by computing explicitly the
Donaldson-Futaki invariant of a test configuration arising from  the Rees algebra for a principal ideal.  These test
configurations, which we call the Rees deformation, are a simplified
version of the deformation to the normal cone test configurations
studied by Ross-Thomas \cite{RT, RT1}. Let $R=
\C[x_{1},\dots,x_{N}]/(f_{1},\dots,f_{d})$, and $Y= \Spec R$ be an
affine variety with an effective, holomorphic action action of a torus
$T$, and let $V \subset Y$ be an invariant subscheme, corresponding to a
homogenous ideal $I \subset R$.  Suppose that $\xi \in \mt$ is a Reeb
vector field.  We consider the Rees algebra of $R$ with respect to $I$,
given by
\begin{equation}
\mathcal{R}=\mathcal{R}(R,I) := \bigoplus_{n\in\mathbb{Z}}t^{-n}I^{n} =
R[t,t^{-1}I] \subset R[t,t^{-1}]
\end{equation}
where $I^{n} := R$ for $n \leq 0$.  For ease of notation we set
$\mathcal{Y} = \Spec \mathcal{R}$.  Note that $\mathcal{Y}$ admits a
$\C^{*}$ action induced by $\lambda \cdot t = \lambda^{-1} t $ for $\lambda
\in \C^{*}$.  The canonical inclusion $\C[t] \hookrightarrow
\mathcal{R}$ gives a map $\pi:\mathcal{Y} \rightarrow \C$, and this map
is clearly $\C^{*}$ equivariant with respect to the above action.  The
scheme $\mathcal{Y}$ carries a natural action of $T$ by acting on the
$t$-graded components, and hence commuting with the $\C^{*}$ action.
For $\alpha \in \C-\{0\}$, the fibre $\pi^{-1}(\alpha) \cong Y$, as
$\mathcal{R}/(T-\alpha)\mathcal{R} \cong R$, and so the generic fibre is
isomorphic to $Y$.  The $T$ action on $\mathcal{Y}$ clearly preserves
the fibres, and restricts to the action of $T$ on $Y$ away from the
central fibre.  Moreover, we have
\begin{equation*}
\mathcal{Y}_{0} := \pi^{-1}(0) = \Spec \bigoplus_{n\geq0}I^{n}/I^{n+1},
\end{equation*}
and so the central fibre is precisely the normal cone of $V$ in $Y$.
The $\C^{*}$ action on the central fibre is determined by the grading
giving $I^{n}/I^{n+1}$ degree $n$.  Moreover, if $\xi \in \mt$ is the
Reeb field, then $\xi$ induces a Reeb field on $\mathcal{Y}_{0}$.  To
see this, observe that if $\xi$ induces a positive grading on $R$, and
$I\subset R$ is homogeneous, then $\xi$ also induces a positive grading
on
\begin{equation*}
R/I \oplus I/I^{2}\oplus \cdots \oplus I^{n}/I^{n+1} \oplus \cdots
\end{equation*}
Finally, it is well known that $\mathcal{R}(R,I)$ is flat over $\C[t]$;
see for instance \cite{Eis}. 

In order to obtain the Lichnerowicz obstruction, we consider the simplest family Rees deformations; namely, those obtained from principal ideals.
Fix a holomorphic function $f:Y\rightarrow \mathbb{C}$, which is
homogeneous for the torus action.  We denote by $\alpha \in \mt^{*}$ the
weight of $f$ under $T$.   Consider the ideal $I = (f) \subset R$, and
the test configuration given by the Rees algebra $\mathcal{R}(R,I)$.
The central fibre, which we denote by $Y_{0}$, of this test
configuration is determined by the ring
\begin{equation*}
 \bigoplus_{n\geq0}I^{n}/I^{n+1} \cong R/I \otimes_{\C} \C[w].
\end{equation*}
The grading on the latter ring is induced by the torus $T$ on the first
factor.  The torus action on the second factor is by weight $\alpha$ on
$w$.  Finally, the induced $\C^{*}$ action, denoted $\eta$, on the
central fibre is trivial on $R/I$, and acts with weight $1$ on $w$.  We
can compute the Donaldson-Futaki invariant of this test configuration
entirely in terms of the weight of the torus action on $f$, and the
Hilbert series of $R$.  First, we observe that if
$H_{R}(z_{0},\dots,z_{s})$ is the Hilbert series of $R$ with
multigrading induced by $T$, and $\alpha_{0},\dots,\alpha_{s}$ denote
the weight of $f$ under multigrading, then
\begin{equation*}
H_{R/I}(z_{0},\dots,z_{s}) = (1-z_{0}^{\alpha_{0}}\cdots
z_{s}^{\alpha_{s}})H_{R}(z_{0},\dots,z_{s})
\end{equation*}
is the Hilbert series of $R/I$.  This follows immediately from the
degree shifted exact sequence
\begin{equation*}
0\longrightarrow
R^{[\alpha_{0},\dots,\alpha_{s}]}\stackrel{f}{\longrightarrow} R
\longrightarrow R/I\longrightarrow0.
\end{equation*}
Since the Hilbert series is multiplicative on tensor products, we have
\begin{equation*}
H_{R/I \otimes_{R} \C[w]}(z_{0},\dots, z_{s}, \tilde{z}) =\frac{
(1-z_{0}^{\alpha_{0}}\cdots
z_{s}^{\alpha_{s}})}{(1-z_{0}^{\alpha_{0}}\cdots
z_{s}^{\alpha_{s}}\tilde{z})}H_{R}(z_{0},\dots,z_{s}).
\end{equation*}
Suppose that the index character of $Y$ with Reeb field $\xi \in \mt$ expands as
\begin{equation*}
F(\xi,t) = \frac{a_{0}(\xi)n!}{t^{n+1}} + \frac{a_{1}(\xi)(n-1)!}{t^{n}} + O(t^{1-n}).
\end{equation*}
then one easily obtains that the index character of the central fibre is given by
\begin{equation*}
\begin{aligned}
F(\xi-s\eta,t) &=
\frac{1-e^{-t\alpha(\xi)}}{(1-e^{-t(\alpha(\xi)-s)})}\left(\frac{a_{0}(\xi)n!}{t^{n+1}}
+ \frac{a_{1}(\xi)(n-1)!}{t^{n}} + O(t^{1-n})\right)\\
&=\frac{a_{0}(\xi)\alpha(\xi)n!}{(\alpha(\xi)-s)t^{n+1}}
+\frac{\alpha(\xi)(n-1)!}{(\alpha(\xi)-s)t^{n}}
\left[a_{1}(\xi)-\frac{s}{2}a_{0}(\xi)n\right]+\dots.
\end{aligned}
\end{equation*}
Applying Theorem~\ref{equivariant RR}, the Donaldson-Futaki invariant is
given by
\begin{equation}\label{DF Rees}
Fut(Y_{0},\xi,\eta) =
\frac{-1}{n(n+1)}\left[\frac{a_{1}(\xi)}{\alpha(\xi)} 
-\frac{n(n+1)}{2}a_{0}(\xi)\right].
\end{equation}
Until now, our developments have been completely general, and
equation~(\ref{DF Rees}) is the formula for the Donaldson-Futaki
invariant of the Rees algebra for a homogeneous principal ideal.  We now
employ the assumption that the $Y$ is Gorenstein and Calabi-Yau, and
$\Theta \in H^{0}(X, K_{X})$ is a non-vanishing section satisfying
$\mathcal{L}_{\xi}\Theta = i(n+1)\Theta$; equation~(\ref{Gore relation})
applies, and so
\begin{equation*}
Fut(Y_{0},\xi,\eta)= -\frac{1}{2}\left[\frac{1}{\alpha(\xi)}-1\right].
\end{equation*} 
In particular, we have the following theorem, which was proved for
rational Reeb vector fields in \cite{RT}.
\begin{thm}\label{L thm}
Let $(Y,\Theta)$ be an isolated Gorenstein singularity with link $L$,
and Reeb vector field $\xi$ satisfying $\mathcal{L}_{\xi}\Theta=
i(n+1)\Theta$.  If $Y$ admits a holomorphic function $f$ with
$\mathcal{L}_{\xi}f=i\lambda f$, and $\lambda <1$, then $(Y,\xi)$ is
K-unstable.
\end{thm}
This gives the following corollary, which was first observed in \cite{GMSY}.
\begin{cor}
Let $(Y,\Theta)$ be an isolated Gorenstein singularity with link $L$,
and Reeb vector field $\xi$ satisfying $\mathcal{L}_{\xi}\Theta=
i(n+1)\Theta$.  If $Y$ admits a holomorphic function $f$ with
$\mathcal{L}_{\xi}f=i\lambda f$, and $\lambda <1$, then $(Y,\xi)$ does
not admit a compatible constant scalar curvature K\"ahler metric.  In
particular, $L$ does not admit a Sasaki-Einstein metric with Reeb field
$\xi$.
\end{cor}

Note that even if $Y$ is not a Gorenstein singularity, from \eqref{DF
Rees} we obtain a lower bound on $\alpha(\xi)$ in terms of the ratio
$a_1/a_0$ whenever $\xi$ is a K-semistable Reeb field on $Y$.

\end{document}